\documentclass{amsart}
\usepackage[utf8]{inputenc}
\usepackage[utf8]{inputenc}
\usepackage{amsthm}
\theoremstyle{definition} 
\theoremstyle{plain} \newtheorem{lemma}{Lemma}[section]
\theoremstyle{plain} 
\theoremstyle{plain} \newtheorem{theorem}{Theorem}[section]
\theoremstyle{plain} \newtheorem{corollary}{Corollary}[section]
\theoremstyle{plain} 
\theoremstyle{plain} 
\theoremstyle{plain} 
\theoremstyle{remark} 
\theoremstyle{remark} 
\theoremstyle{plain} 
\theoremstyle{remark} \newtheorem{remark}{Remark}[section]
\usepackage{verbatim}
\usepackage{textcomp}
\usepackage{latexsym}
\usepackage{amsfonts}
\usepackage{amssymb}
\usepackage{amsmath}
\usepackage{mathabx}
\usepackage{mathrsfs}
\usepackage{systeme}
\DeclareMathOperator{\Ima}{im}
\numberwithin{equation}{section}

\newcommand{\namedthm}[2]{\theoremstyle{plain}
   \newtheorem*{thm#1}{#1}\begin{thm#1}#2\end{thm#1}}

\title{Conformally invariant rigidity theorems on four-manifolds with boundary}
\author{Siyi Zhang}
\address{Department of Mathematics, University of Notre Dame, Notre Dame, Indiana 46556, USA}
\email{szhang25@nd.edu}
\begin{document}
\date{\today}
\maketitle
\begin{abstract} 
In the article we introduce new conformal and smooth invariants on compact, oriented four-manifolds with boundary. In the first part, we show that ``positivity'' conditions on these invariants will impose topological restrictions on underlying manifolds with boundary, which generalizes the results on closed four-manifolds by M. Gursky and on conformally compact Einstein four-manifolds by S.-Y. A. Chang, J. Qing, and P. Yang. In the second part, we study Weyl functional on four-manifolds with boundary and establish several conformally invariant rigidity theorems. As applications, we prove some rigidity theorems for conformally compact Einstein four-manifolds. These results generalize the work on closed four-manifolds by S.-Y. A. Chang, J. Qing, and P. Yang and rigidity theorem for conformally compact Einstein four-manifolds by G. Li, J. Qing, and Y. Shi. A crucial idea of the proofs is to understand the expansion of a smooth Riemannian metric near the boundary. It is noteworthy to point out that we rule out some examples arising from the study of closed manifolds in the setting of manifolds with umbilic boundary. 

\end{abstract}

\section{Introduction}

In this article we introduce new conformal and smooth invariants and study the geometry and topology of compact, oriented four-manifolds with boundary by investigating these invariants. The paper is divided into two parts. In the first part, we study the topological implications of the the ``positivity'' conditions on conformal invariants. Our main results are partially motivated by the corresponding works on compact four-manifolds without boundary and conformally compact Einstein four-manifolds. For compact four-manifolds without boundary M. Gursky established the following theorem in \cite{Gur98}:
\namedthm{Theorem A}{(\cite{Gur98}, Corollary F; see also \cite{CGY02} Corollary B) Let $(M^4,g)$ be a closed, oriented four-manifold with 
\begin{align}\label{closed positive}
    Y(M^4,[g]) > 0,\,\,\,\,\int\sigma_2(P_{g})\,dv_{g} > 0.
\end{align}
Then $H^1(M^4) = 0$.}
This result reveals that the positivity condition (\ref{closed positive}) on conformal invariants imposes topological restrictions on the underlying manifold. The Yamabe invariant $Y(M^4,[g])$ of $(M^4,g)$ is defined as
\begin{align}
    {Y}(M^4,[g]) = \inf_{\widetilde{g}\in[g],Vol(\tilde{g}) = 1} \int_MR_{\widetilde{g}}\,dv_{\widetilde{g}},
\end{align}
where $R_g$ is the scalar curvature of $g$ and $[g] = \{ e^{2f}g \, : \, f\in{C^{\infty}(M)} \}$ denotes the conformal class of $(M^4,g)$. It is clear that $Y(M^4,[g])$ is a conformal invariant. 

In four dimensions, $P_{g}$ is the Schouten tensor defined as
\begin{align}
    P = \frac{1}{2}\left(Ric - \frac{1}{6}Rg\right),
\end{align} 
and $\sigma_k(P)$ is the $k$-th elementary symmetric polynomial applied to the eigenvalues of $g^{-1}P$. In addition, we have the Chern-Gauss-Bonnet formula:
\begin{align}\label{closed CGB}
    8\pi^2\chi(M^4) = \int_{M}||W_g||^2\,dv_{g} + 4\int_{M}\sigma_2(P)\,dv_{g},
\end{align}
where $W_{g}$ is the Weyl curvature and $\chi(M^4)$ is the Euler characteristic of $M^4$. It is well-known that $||W_g||^2\,dv_{g}$ is a pointwise conformal invariant in four dimensions and $\chi(M^4)$ is a topological invariant. It follows from (\ref{closed CGB}) that $\int_{M}\sigma_2(P)\,dv_{g}$ is a conformal invariant. 

In order to state our main results, we need to establish some additional notation. From now on, we assume $(M^4,\partial{M^4} = \Sigma^3,g)$ is a compact, oriented Riemannian four-manifold with boundary. Denote by $L$ and $H$ the second fundamental form and mean curvature of the boundary, respectively. In this note, we shall call the boundary $\Sigma^3$ umbilic if $L_{ij} = \mu g_{ij}$ for some smooth function $\mu$ on $\Sigma^3$ and totally geodesic if $L_{ij} \equiv 0$ on $\Sigma^3$. The first Yamabe invariant of $(M^4,\Sigma^3,g)$ is defined as
\begin{align}
    {Y}(M^4,\Sigma^{3},[g]) = \inf_{\widetilde{g}\in[g],Vol(\tilde{g}) = 1} \left(\int_MR_{\widetilde{g}}\,dv_{\widetilde{g}}+2\int_{\Sigma}H_{\widetilde{g}}\,d\sigma_{\widetilde{g}}\right).
\end{align}
By the work of J. Escobar \cite{Esc92}, ${Y}(M^4,\Sigma^{3},[g])$ is always attained by a metric $g_{Y} \in [g]$ with constant scalar curvature and minimal boundary. If in addition the boundary is umbilic, then the boundary is totally geodesic. See Section \ref{prelim} for more details. On a compact Riemannian four-manifold with boundary, we have the following Chern-Gauss-Bonnet formula :
\begin{align}\label{CGB}
    8\pi^2\chi(M^4,\Sigma^3)=\int_M||W_g||^2dv_g+4\left(\int_M\sigma_2(P_g)\,dv_g+\frac{1}{2}\int_{\Sigma}\mathcal{B}_g\,d\sigma_g\right),
\end{align}
where
\begin{align}
    \mathcal{B}_g = \frac{1}{2}R_gH - R_{00}H - R_{kikj}L^{ij} + \frac{1}{3}H^3 - H|L|^2 + \ \frac{2}{3}trL^3,
\end{align}
$L$ and $H$ are the second fundamental form and mean curvature of $\Sigma^3$, Latin letters run through $1,2,3$ as tangential directions, and $0$ is the outward normal direction on $\Sigma^3$. 
In place of $\int_M\sigma_2(P_g)\,dv_g$ we have the conformal invariant \cite{Chen09}
\begin{align}
    \mathcal{E}([g]) := \int_M\sigma_2(P_g)\,dv_g+\frac{1}{2}\int_{\Sigma}\mathcal{B}_g\,d\sigma_g.
\end{align}
We remark that $\mathcal{E}([g])$ is the same as $\int_M\sigma_2(P_g)\,dv_g$ when the boundary is totally geodesic since $ \mathcal{B}_g \equiv 0$. We now introduce the following conformal classes on a compact four-manifold with boundary $(M^4,\Sigma^3)$:
\begin{align} \label{Y1}
\mathcal{Y}_{1,b}^{+}(M^4,\Sigma^3) = \{ \, g : Y(M^4,\Sigma^3, [g]) > 0 \, \},
\end{align}
and
\begin{align} \label{Y2def}
\mathcal{Y}_{2,b}^{+}(M^4,\Sigma^3) = \{ \, g \in \mathcal{Y}_{1,b}^{+}(M^4,\Sigma^3) : \mathcal{E}([g]) > 0 \, \}.
\end{align}
It is now clear that the condition $g \in \mathcal{Y}_{2,b}^{+}(M^4,\Sigma^3)$ can be viewed as a generalization of the condition (\ref{closed positive}) in Theorem A. 

Our following result generalizes Theorem A to the setting of four-manifolds with boundary. 
\begin{theorem}\label{Connect}
Let $(M^4,\partial{M^4} = \Sigma^3)$ admit a metric in $\mathcal{Y}_{2,b}^+(M^4,\Sigma^3)\ne\emptyset$ with umbilic boundary. Then $H^1(M^4,\Sigma^3) = H^1(M^4) = 0$. In addition, the boundary $\Sigma^3$ is connected.
\end{theorem}

We now turn to conformally compact Einstein four-manifolds. On a conformally compact Einstein four-manifold $(M^4,g_{+})$, denote by $V$ its renormalized volume which appears in M. Anderson's formula \cite{And01}:
\begin{align}\label{CGB And}
    8\pi^2\chi(M^4,\Sigma^3) = \int_M ||W||^2 \,dv + 6V.
\end{align}
See Section 2 for more details. In \cite{CQY04}, S.-Y. A. Chang, J. Qing, and P. Yang proved the following two theorems: 
\namedthm{Theorem B}{(\cite{CQY04}, Theorem A)
Let $(M^4,g_{+})$ be a conformally compact Einstein four-manifold with conformal infinity of positive Yamabe type. Then
\begin{align}
    V > \frac{1}{3}\frac{4\pi^2}{3}\chi(M)
\end{align}
implies that $M^4$ is a homology $B^4$.
}
\namedthm{Theorem C}{(\cite{CQY04}, Theorem B)
Let $(M^4,g_{+})$ be a conformally compact Einstein four-manifold with conformal infinity of positive Yamabe type. Then
\begin{align}
    V > \frac{1}{2}\frac{4\pi^2}{3}\chi(M)
\end{align}
implies that $M^4$ is diffeomorphic to $B^4$ and the conformal infinity $\Sigma^3$ of $(M^4,g_{+})$ is diffeomorphic to $S^3$.}
Next we obtain by combining (\ref{CGB}) and (\ref{CGB And}) 
\begin{align}
    \mathcal{E}([g]) = \frac{3}{2}V.
\end{align}
Hence, the comparison between $\chi$ and $V$ can be rewritten as comparison between $\int_M ||W_{g}||^2 \,dv_{g}$ and $\mathcal{E}([g])$. Partially inspired by this observation, for metrics $g \in \mathcal{Y}_{2,b}^{+}(M^4,\Sigma^3)$ we introduce a conformal invariant
\begin{align} 
\beta_{b}(M^4,\Sigma^3,[g]) = \dfrac{ \int_{M} \| W_g \|^2 \,dv_g }{\mathcal{E}([g])} \geq 0
\end{align}
and a smooth invariant
\begin{align}
    \beta_{b}(M^4,\Sigma^3) = \inf_{[g]}\beta_{b}(M^4,\Sigma^3,[g]).
\end{align}
If $\mathcal{Y}_{2,b}^{+}(M^4,\Sigma^3) = \emptyset$, set $\beta_{b}(M^4,\Sigma^3) = -\infty$. With these preliminaries, we state two theorems which generalize Theorems B and C to the setting of compact, oriented four-manifolds with boundary.
\begin{theorem}\label{hom}
Let $(M^4,\partial{M^4} = \Sigma^3,g)$ satisfy $0 \leq \beta_b(M^4,\Sigma^3,[g]) < 8$ with umbilic boundary. Then the double of $(M^4,\Sigma^3)$ is homeomorphic to $S^4$, $\Sigma^3$ is a homology $S^3$, and $M^4$ is a homology $B^4$.
\end{theorem}
\begin{theorem}\label{dif}
Let $(M^4,\partial{M^4} = \Sigma^3,g)$ satisfy $0 \leq \beta_b(M^4,\Sigma^3,[g]) < 4$ with umbilic boundary. Then $M^4$ is diffeomorphic to $B^4$ and $\Sigma^3$ is diffeomorphic to $S^3$.
\end{theorem}
We remark that any compactification of a conformally compact Einstein manifold has umbilic boundary. It is noteworthy to point out that in conformally compact Einstein case, one only needs to require the conformal infinity have positive Yamabe constant, while for Theorems \ref{hom} and \ref{dif} one does need the positivity of ${Y}(M^4,\Sigma^3,[g])$. See Section 2 for more details. Note that $\mathbb{CP}^2$ \emph{cannot} be realized as the double in Theorem \ref{hom}, which is in contrast with the closed case shown by Theorem A of \cite{CGZ}. It reveals that the umbilic condition on the boundary imposes additional symmetry on the double manifold. 

In the second part, we study the critical points of Weyl functional on compact, oriented four-manifolds with boundary. In particular, we shall establish rigidity theorems for a class of metrics satisfying conformally invariant conditions on compact, oriented four-manifolds with boundary. Our results are partially motivated by the following rigidity theorem for {closed} Bach-flat four-manifolds proved by S.-Y. A. Chang, J. Qing, and P. Yang in \cite{CQY}:
\namedthm{Theorem D}{(\cite{CQY})
Suppose $(M^4,g)$ is an oriented, closed Bach-flat four-manifold and $g\in\mathcal{Y}_2^+(M^4)$ satisfies
\begin{align}\label{beta sphere}
    0\leq\beta(M^4,[g])<4,
\end{align}
then $(M^4,g)$ is conformally equivalent to $(S^4,g_{S^4})$, where $g_{S^4}$ is the round metric. In fact, $(M^4,g_{Y})$ is isometric to $(S^4,g_{S^4})$ where $g_{Y}\in[g]$ is the Yamabe metric.}
This result shows a conformally invariant rigidity phenomenon for closed Bach-flat four-manifolds. The Weyl functional on closed four-manifolds is defined by
\begin{align*}
    \mathcal{W} \, : \, g \, \mapsto \int||W_{g}||^2 \, dv_{g}
\end{align*}
where $W_g$ denotes the Weyl curvature tensor of $g$ and $\| \cdot \|$ is the norm of $W_{g}$ as a section of $End(\Lambda^2(M))$. Critical points of $\mathcal{W}$ are metrics with vanishing Bach tensor $B_{\alpha \beta}$ defined by 
\begin{align} \label{BF} 
B_{\alpha\beta} = \nabla^{\gamma}\nabla^{\delta} W_{\alpha\gamma\beta\delta} + P^{\gamma\delta}W_{\alpha\gamma\beta\delta},
\end{align}
where $P$ is the Schouten tensor. Four-manifolds with vanishing Bach tensor are called {\em Bach-flat} manifolds. It is noteworthy to point out that Bach-flatness is a conformally invariant condition.

Since Bach-flat metrics are critical points of the Weyl functional, Theorem D can be restated in the following way:
\namedthm{Theorem E}
{The round sphere $(S^4,g_{S^4})$ is the unique critical point of the Weyl functional (up to conformal equivalence) satisfying
\begin{align*}
    0 \leq \beta(M^4,[g]) < 4.
\end{align*}
}
Replacing the condition on $\beta$ by the condition on $\int_{M}\sigma_2(P)\,dv_g$, Chang, Qing, and Yang proved the following gap theorem in \cite{CQY} (see also \cite{LQS17} for a simplified and refined proof):
\namedthm{Theorem F}{(\cite{CQY})
There is an $\epsilon > 0$ such that if $(M^4,g)$ is a closed, oriented Bach-flat four-manifold and $g\in\mathcal{Y}_2^+(M^4)$ satisfies
\begin{align}\label{integral sphere}
    \int\sigma_2(P)\,dv_{g} \geq 4(1-\epsilon)\pi^2,
\end{align}
then $(M^4,g)$ is conformally equivalent to $(S^4,g_{S^4})$, where $g_{S^4}$ is the round metric. In fact, $(M^4,g_{Y})$ is isometric to $(S^4,g_{S^4})$ where $g_{Y}\in[g]$ is the Yamabe metric.
}

We shall generalize Theorems D, E, and F to the setting of compact Bach-flat four-manifolds with boundary. In order to state our main results we first establish some additional notation. 
 
From now on, we assume that $(M^4,\Sigma^3 = \partial M^4,g)$ is a compact, oriented, four-dimensional Riemannian manifold with boundary. In place of $\mathcal{W}$, consider the functional
\begin{align*}
    \mathcal{W}_b : g \mapsto \int_{M^4}||W_{g}||^2\, dv_g + 2\int_{\Sigma^3}W_{i0j0}L^{ij}\, d{\sigma}_g,
\end{align*}
where Latin letters run through $1,2,3$ as tangential directions, $0$ is the outward normal direction on $\Sigma^3$ and $L$ is the second fundamental form of $\Sigma^3$. This functional generalizes the Weyl functional $\mathcal{W}$ for closed four-manifolds. As pointed out in \cite{CG18}, critical points of $\mathcal{W}_b$ are Bach-flat metrics such that the tensor
\begin{align} \label{S}
S_{ij} := \nabla^{\alpha}W_{{\alpha}i0j} + \nabla^{\alpha}W_{{\alpha}j0i} - \nabla^0W_{0i0j} + \frac{4}{3}HW_{0i0j}
\end{align}
vanishes on the boundary. In this case, we will say that the boundary is $S$-flat. We remark that $\mathcal{W}$ and $\mathcal{W}_b$ are conformally invariant, hence Bach-flatness and $S$-flatness are conformally invariant conditions.

As pointed out in \cite{GZ20}, since the Bach-flat condition is fourth order in the metric, it should be possible to specify a boundary condition in addition to $S$-flatness. We shall say the boundary is umbilic if $L_{ij} = \lambda{g_{ij}}$ for some smooth function $\lambda$ on $\Sigma^3$ and totally geodesic if $L_{ij} \equiv 0$ on $\Sigma^3$. In \cite{GZ20}, the authors proved the following result:
\namedthm{Theorem G}
{(\cite{GZ20})Given a compact four-dimensional manifold with boundary $(M^4, \Sigma^3 = \partial M^4)$, let $\mathcal{M}^{0}(M^4,\Sigma^3)$ denote the space of all Riemannian metrics on $(M^4,\Sigma^3)$ such that $\Sigma^3$ is umbilic.  Then $g$ is a critical point of 
\begin{align*}
\mathcal{W} \big|_{\mathcal{M}^{0}(M^4,\Sigma^3)},
\end{align*}
if and only if $g$ is Bach-flat and $\Sigma^3$ is $S$-flat and umbilic. 
}  
Note that umbilic condition is also conformally invariant. Hence, it is natural to consider Bach-flat four-manifolds with umbilic and $S$-flat boundary as a conformally invariant generalization of closed Bach-flat four-manifolds. 

We now state our main result which generalizes Theorem D.
\begin{theorem}\label{beta b sphere}
Let $(M^4,\Sigma^3,g)$ be a Bach-flat four-manifold with boundary such that the boundary is $S$-flat and umbilic and $g\in\mathcal{Y}_{2,b}^+(M^4,\Sigma^3)$ satisfies
\begin{align}
    0 \leq \beta_b(M^4,\Sigma^3,[g]) < 4.
\end{align}
Then $(M^4,\Sigma^3,g)$ is conformally equivalent to $(S^4_+,S^3,g_{S_{+}^4})$, where $g_{S_{+}^4}$ is the round metric on upper hemisphere.
\end{theorem}

If we replace the condition on $\beta_b$ by the condition on the conformal invariant $\mathcal{E}([g])$, we establish the following generalization of Theorem F.

\begin{theorem}\label{conformal E sphere}
There is an $\epsilon_1 > 0$ such that if $(M^4,\Sigma^3,g)$ is a Bach-flat four-manifold with boundary such that the boundary is $S$-flat and umbilic and $g\in\mathcal{Y}_2^+(M^4,\Sigma^3)$ satisfies
\begin{align}\label{integral sphere boundary}
    \mathcal{E}([g]) \geq 2(1-\epsilon_1)\pi^2,
\end{align}
then $(M^4,\Sigma^3,g)$ is conformally equivalent to $(S^4_+,S^3,g_{S^4_+})$, where $g_{S_{+}^4}$ is the round metric on upper hemisphere.
\end{theorem}

If we replace the condition on $\beta_b$ by the condition on the functional $\mathcal{W}_b$, we establish the following conformally invariant rigidity theorem.

\begin{theorem}\label{beta b and Weyl sphere}
Let $(M^4,\Sigma^3,g)$ be a Bach-flat four-manifold with boundary such that the boundary is $S$-flat and umbilic and $g\in\mathcal{Y}_{2,b}^+(M^4,\Sigma^3)$ satisfies
\begin{align}
    \mathcal{W}_{b}([g]) < 4\pi^2.
\end{align}
Then $(M^4,\Sigma^3,g)$ is conformally equivalent to $(S^4_+,S^3,g_{S_{+}^4})$, where $g_{S_{+}^4}$ is the round metric on upper hemisphere.
\end{theorem}

As pointed out by Theorem 1.2, for four-manifolds with umbilic boundary, the ``critical'' value for $\beta_b$ is actually $8$. With additional Bach-flatness and $S$-flatness, we prove the following theorem:
\begin{theorem}\label{beta b 8}
Let $(M^4,\Sigma^3,g)$ be a Bach-flat four-manifold with boundary such that the boundary is $S$-flat and umbilic and $g\in\mathcal{Y}_{2,b}^+(M^4,\Sigma^3)$. There is an $\epsilon_2 >0$ such that if $g$ satisfies
\begin{align}
    0 \leq \beta_b(M^4,\Sigma^3,[g]) < 8(1+\epsilon_2),
\end{align}
then the double of $(M^4,\Sigma^3)$ is homeomorphic to $S^4$, $\Sigma^3$ is a homology $S^3$, and $M^4$ is a homology $B^4$.
\end{theorem}

Next we turn to the applications of Theorems \ref{beta b sphere}, \ref{conformal E sphere}, and \ref{beta b and Weyl sphere} to conformally compact Einstein four-manifolds. We first recall the boundary expansion for a conformally compact Einstein four-manifold. Let $(M^4,\Sigma^3,g)$ be the compactification of a conformally compact Einstein four-manifold $(M^4,g_+)$ with geodesic defining function $r$ and $(\Sigma^3,[h])$ as the conformal infinity, where $h = g|_{\Sigma}$. Given $h$ in the conformal infinity we may choose a \emph{geodesic} defining function $r$, such that near the boundary $\bar{g} = r^2g_{+}$ can be written as
\begin{align}\label{PE expansion}
    \bar{g} = dr^2 + h + g^{(2)}r^2 + g^{(3)}r^3 + O(r^4),
\end{align}
where $g^{(2)}$ and $g^{(3)}$ are tensors on $\Sigma^3$. It is noteworthy to point out that (\ref{PE expansion}) shows that any compactification of a conformally compact Einstein four-manifold has umbilic boundary. The item $g^{(3)}$ is not determined by the intrinsic geometry of $(\Sigma^3,h)$ and is referred to as the non-local term.  It turns out that the $S$-tensor can be understood as (up to a constant multiple) the non-local term $g^{(3)}$ in (\ref{PE expansion}) as observed in \cite{CG18}:
\begin{align}\label{PE S tensor}
    S_{ij} = -\frac{3}{2}g^{(3)}_{ij}.
\end{align}
With these preliminaries, we now state several applications of our main theorems to conformally compact Einstein four-manifolds. It is noteworthy to point out that for any compactification $(M^4,\Sigma^3,g)$ of conformally compact Einstein four-manifolds $Y(\Sigma^3,[h])>0$ implies $Y(M^4,\Sigma^3,[g]) > 0$. See Lemma \ref{Qing} in Section \ref{prelim} for more details.

\begin{corollary}\label{PE hyper}
Let  $(M^4,g_{+})$  be a conformally compact Einstein four-manifold with conformal infinity of positive Yamabe type. Then  $(M^4,g_{+})$  is isometric to the four-dimensional  hyperbolic space if it satisfies 
\begin{align}\label{Cor E}
    \int_{M}||W_{g_{+}}||^2\,dv_{g_{+}} < 6V(M^4,g_{+})
\end{align}
and the non-local term $g^{(3)}$ in (\ref{PE expansion}) vanishes.
\end{corollary}

\begin{remark}
Note that (\ref{Cor E}) is equivalent to 
\begin{align}
    V(M^4,g_{+}) > \frac{2\pi^2}{3}\chi(M^4)
\end{align}
by (\ref{CGB And}). Compare with Theorem C.
\end{remark}

\begin{corollary}\label{LQS}
There is an $\epsilon > 0$ such that a conformally compact Einstein four-manifold $(M^4,g_{+})$ with conformal infinity of positive Yamabe type is isometric to the hyperbolic space if its renormalized volume satisfies 
\begin{align}
    V(M^4,g_{+}) \geq (1-\epsilon)\frac{4\pi^2}{3} = (1-\epsilon)V(\mathbb{H}^4,g_{\mathbb{H}^4})
\end{align}
and the non-local term $g^{(3)}$ in (\ref{PE expansion}) vanishes, where $g_{\mathbb{H}^4}$ is the hyperbolic metric.
\end{corollary}
\begin{remark}
Corollary \ref{LQS} recovers the rigidity theorem proved by G. Li, J. Qing, and Y. Shi in \cite{LQS17}.
\end{remark}

\begin{corollary}\label{PE Weyl}
Let  $(M^4,g_{+})$  be a conformally compact Einstein four-manifold with conformal infinity of positive Yamabe type. Then  $(M^4,g_{+})$  is isometric to the four-dimensional hyperbolic space if it satisfies 
\begin{align}\label{Cor F}
    V(M^4,g_{+}) > 0, \,\,\,\,\,\, \int_{M}||W_{g_{+}}||^2dv_{g_{+}} < 4\pi^2,
\end{align}
and the non-local term $g^{(3)}$ in (\ref{PE expansion}) vanishes.
\end{corollary}

\begin{remark}
It is not hard to see that Corollaries \ref{PE hyper}, \ref{LQS}, and \ref{PE Weyl} would not be valid without the vanishing condition of $g^{(3)}$, which is shown by the examples constructed in \cite{GL91}.
\end{remark}

\begin{remark}
According to the author's knowledge, there is currently no example of a conformally compact Einstein four-manifold with conformal infinity of positive Yamabe type and vanishing non-local term other than the four-dimensional hyperbolic space. 
\end{remark}

The paper is organized as follows. In Section \ref{prelim} we establish the notations and collect preliminaries. In Section \ref{1.1,1.2,1.3} we prove Theorems \ref{Connect}, \ref{hom}, and \ref{dif} concerning the topological properties of the underlying manifolds under conformally invariant conditions. In Section \ref{1.4,1.5,1.6,1.7} we prove Theorems \ref{beta b sphere}, \ref{conformal E sphere}, \ref{beta b and Weyl sphere}, and \ref{beta b 8} concerning the rigidity of compact Bach-flat four-manifolds with boundary. In Section \ref{final} we record some further remarks.

\section*{acknowledgement}
The author would like to thank Professors Sun-Yung A. Chang and Paul Yang for suggesting the problem and Professor Matthew Gursky for numerous helpful comments and discussions.

\section{Preliminaries}\label{prelim}
\subsection{Conformal and smooth invariants on four-manifolds with boundary}

Let $(M^n,\partial{M^{n}} = \Sigma^{n-1}, g)$ be a compact Riemannian manifold with boundary. The first Yamabe invariant of $(M^n,\Sigma^{n-1}, g)$ is defined as
\begin{align}
    {Y}(M^n,\Sigma^{n-1},[g]) = \inf_{\widetilde{g}\in[g],Vol(\tilde{g}) = 1} \left(\int_MR_{\widetilde{g}}\,dv_{\widetilde{g}}+2\int_{\Sigma}H_{\widetilde{g}}\,d\sigma_{\widetilde{g}}\right).
\end{align}

Any smooth metric in a conformal class attaining this infimum has constant scalar curvature and minimal boundary. From the work of J. Escobar \cite{Esc92}, it is known that in many cases, such a minimizer exists. In particular, for $3\leq{n}\leq{5}$, a minimizer always exists. In addition, Escobar established the following inequality in these dimensions:
\begin{align}
    {Y}(M^n,\Sigma^{n-1},[g])\leq{Y}(S^n_+,S^{n-1},[g_{S_{+}^n}]),
\end{align}
where equality holds if and only if $(M^n,\Sigma^{n-1},g)$ is conformally equivalent to the round hemispehre $(S_+^n,S^{n-1},g_{S_{+}^n})$.

Returning to the Chern-Gauss-Bonnet formula (\ref{CGB}) in four dimensions, we have mentioned that
\begin{align}
    \mathcal{E}([g]) := \int_M\sigma_2(P_g)\,dv_g+\frac{1}{2}\int_{\Sigma}\mathcal{B}_g\,d\sigma_g
\end{align}
is a conformal invariant. The invariant $\mathcal{E}([g])$ has been studied extensively in \cite{Chen09}. In particular, if $Y(M^4,\Sigma^3,[g]) > 0$ and the boundary is umbilic, then we have
\begin{align}
    \mathcal{E}([g]) \leq 2\pi^2,
\end{align} 
where equality holds if and only if $(M^4,\Sigma^{3},g)$ is conformally equivalent to the round hemispehre $(S_+^4,S^{3},g_{S_{+}^4})$. We remark that $\mathcal{B}_g \equiv 0$ if the boundary is totally geodesic. In this case, we have
\begin{align}
    \mathcal{E}([g]) = \int_M\sigma_2(P_g)\,dv_g.
\end{align}
Therefore, $\mathcal{E}([g])$ is a natural generalization of total integral of $\sigma_2$-curvature on four-manifolds with boundary.

Inspired by the study of conformal geometry on closed four-manifolds \cite{CGZ}, we now introduce the following conformal classes on a compact four-manifold with boundary $(M^4,\Sigma^3)$:
\begin{align} 
\mathcal{Y}_{1,b}^{+}(M^4,\Sigma^3) = \{ \, g : Y(M^4,\Sigma^3, [g]) > 0 \, \},
\end{align}
and
\begin{align} 
\mathcal{Y}_{2,b}^{+}(M^4,\Sigma^3) = \{ \, g \in \mathcal{Y}_{1,b}^{+}(M^4,\Sigma^3) : \mathcal{E}([g]) > 0 \, \}.
\end{align}
In addition, we introduce a conformal invariant for metrics $g \in \mathcal{Y}_{2,b}^{+}(M^4,\Sigma^3)$
\begin{align} 
\beta_{b}(M^4,\Sigma^3,[g]) = \dfrac{ \int_{M} \| W_g \|^2 \,dv_g }{\mathcal{E}([g])} \geq 0
\end{align}
and a smooth invariant
\begin{align}
    \beta_{b}(M^4,\Sigma^3) = \inf_{[g]}\beta_{b}(M^4,\Sigma^3,[g]).
\end{align}
If $\mathcal{Y}_{2,b}^{+}(M^4,\Sigma^3) = \emptyset$, set $\beta_{b}(M^4,\Sigma^3) = -\infty$. We remark that $\beta_{b}(M^4,\Sigma^3,[g])$ and $\beta_{b}(M^4,\Sigma^3)$ can be viewed as natural generalizations of $\beta(M^4,[g])$ and $\beta(M^4)$ which are defined on closed four-manifolds in \cite{CGZ}.

\subsection{Conformally compact Einstein manifolds}
In this subsection, we recall some basic notions for conformally compact Einstein manifolds. Suppose $X$ is the interior of a smooth, compact manifold with boundary $(\bar{X}, Z =\partial{X})$. A metric $g_{+}$ defined in $X$ is conformally compact if there is a defining function for the boundary $\rho \, : \, \bar{X} \to \mathbb{R}$ such that $\bar{g} = \rho^2{g_{+}}$ defines a metric on $\bar{X}$. By a defining function, we mean a smooth function with $\rho > 0$ in $X$, $\rho = 0$ and $d\rho \ne 0$ on $\partial{X}$. We will assume in the following that $\bar{g}$ is at least $C^2$ up to the boundary. If $(X,\partial{X}, g_{+})$ is Einstein, then we say that $(X,\partial{X},g_{+})$ is a conformally compact Einstein (CCE) manifold. The choice of defining function is not unique, and thus a conformally compact manifold $(X, \partial{X}, g_{+})$ naturally defines a conformal class of metrics on the boundary, $[h]$, called the conformal infinity.

Suppose $(X^4,Z^3 =\partial{X^4}, g_{+})$ is a conformally compact Einstein four-manifold. Given $h$ in the conformal infinity we may choose a \emph{geodesic} defining function $r$, such that near the boundary $\bar{g} = r^2g_{+}$ can be written as
\begin{align}
    \bar{g} = dr^2 + h + g^{(2)}r^2 + g^{(3)}r^3 + O(r^4),
\end{align}
where $g^{(2)}$ and $g^{(3)}$ are tensors on $Z^3$. In order to give an explanation of AdS/CFT correspondence, M. Henningson and K. Skenderis in \cite{HenSken} defined and calculated the renormalized volume for a conformally compact Einstein manifold. See \cite{Graham} for a detailed mathematical exposition of the renormalized volume. In dimension four, they considered the expansion
\begin{align}
    Vol_{g_{+}}(\{\rho > \epsilon\}) = c_{0}\epsilon^{-3} + c_{2}\epsilon^{-1} + V + o(1),
\end{align}
where $c_0$ and $c_2$ are integrals of local scalar invariants on $Z^3$. More importantly, the constant item $V$ is independent of the choice of defining function and is called the \emph{renormalized volume} of $(X^4,g_{+})$.

Ton conclude this subsection we recall a result of J. Qing in \cite{Qing03} which we specify in four dimensions: 
\begin{lemma}\label{Qing}
Let $(X^4,g_{+})$ be a conformally compact Einstein four-manifold and denote by $(Z^3,[h])$ its conformal infinity. If 
\begin{align}
    Y(Z^3,[h]) > 0,
\end{align}
then we have 
\begin{align}
    Y(X^4,Z^3,[\bar{g}]) > 0,
\end{align}
where $\bar{g} = r^2g_{+}$ is a conformal compactification.
\end{lemma}
By Lemma \ref{Qing}, it is clear that Theorems B and C can be viewed as applications of Theorems \ref{hom} and \ref{dif} to conformally compact Einstein four-manifolds. 

\subsection{Bach-flat metrics on four-manifolds with boundary} In this subsection, we discuss Riemannian functionals and their critical points on four-manifolds with boundary. We will focus on Bach-flat metrics with conformally invariant boundary conditions.

We start with a concise review on Riemannian functionals on {closed} manifolds. See Chapter 4 of \cite{Bes87} for a detailed description of Riemannian functionals. Suppose $M^n$ is a closed, smooth manifold and denote by $\mathcal{M}$ the set of Riemannian metrics on $M$. A Riemannian functional is a real-valued function $\mathcal{F}$ on $\mathcal{M}$ such that $\mathcal{F}(\varphi^{*}(g)) = \mathcal{F}(g)$ for every diffeomorphism $\varphi$ and $g\in\mathcal{M}$. It turns out that Riemannian functionals play an important role in the study of geometry and topology of manifolds. A prominent example is the (normalized) Hilbert-Einstein functional:
\begin{align}
   \mathcal{F}:\,\,\, g \,\,\, \mapsto \,\,\, Vol(g)^{-\frac{n-2}{n}}\int_MR_{g}\, dv_g.
\end{align}
Critical points of $\mathcal{F}$ are Einstein metrics. In four dimensions, quadratic functionals are of particular interest. An important example is the Weyl functional:
\begin{align}
   \mathcal{W}:\,\,\, g \,\,\, \mapsto \,\,\, \int_{M}||W_{g}||^2\, dv_g.
\end{align}
Critical points of $\mathcal{W}$ are Bach-flat metrics. In four dimensions, the Bach tensor is defined  as
\begin{align}
   B_{\alpha\beta} = \nabla^{\gamma}\nabla^{\delta} W_{\alpha\gamma\beta\delta} + P^{\gamma\delta}W_{\alpha\gamma\beta\delta}.
\end{align}
A metric is called Bach-flat if its Bach tensor is vanishing identically. Note that the Weyl functional $\mathcal{W}$ is conformally invariant in four dimensions in the sense that $\mathcal{W}(\widetilde{g}) = \mathcal{W}(g)$ for any $\widetilde{g}\in[g]$. The moduli spaces of Einstein metrics and Bach-flat metrics have been studied extensively; see \cite{And}\cite{AC91}\cite{TV05a}\cite{TV05b}. From the viewpoint of calculus of variations, it is important to understand the rigidity and stability of a critical metric. In this direction, significant progress has been made in \cite{CQY}\cite{GV15}.

Next we turn to Riemannian functionals on compact manifolds with boundary. We shall focus on two important examples.

The (normalized) Hilbert-Einstein functional on $(M^n, \Sigma^{n-1},g)$ is defined as:
\begin{align}
    \mathcal{F}_{b}:\,\,\, g\,\,\, \mapsto\,\,\, Vol(g)^{-\frac{n-2}{n}}\left(\int_MR_{g}\, dv_g + 2\int_{\Sigma}H_g\, d{\sigma}_g\right),
\end{align}
where $H$ is the mean curvature of $\Sigma^{n-1}$. Critical points of $\mathcal{F}_{b}$ are Einstein metrics with totally geodesic boundary \cite{Araujo03}. Note that the infimum of the restriction of normalized Hilbert-Einstein functional in a conformal class gives the first Yamabe invariant.
     
The Weyl functional on $(M^4, \Sigma^{3},g)$ is defined as:
\begin{align}
       \mathcal{W}_b:\,\,\, g \,\,\, \mapsto \,\,\, \int_{M^4}||W_{g}||^2\, dv_g + 2\int_{\Sigma^3}W_{i0j0}L^{ij}\, d{\sigma}_g,
\end{align}
where $L$ is the second fundamental form of $\Sigma^3$, Latin letters run through $1,2,3$ as tangential directions, and $0$ is the outward normal direction on $\Sigma^3$. We remark that $\mathcal{W}_{b} = \mathcal{W}$ on four-manifold with umbilic boundary. Therefore, $\mathcal{W}_b $ can be viewed as a natural generalization of Weyl functional on four-manifolds with boundary since $W_{i0j0}L^{ij} \equiv 0$ on umbilic boundary. The functional $\mathcal{W}_b$ has been studied extensively in \cite{CG18}. It is noteworthy to point out that $\mathcal{W}_b$ is conformally invariant in four dimensions in the sense that $\mathcal{W}_b(\widetilde{g}) = \mathcal{W}_b(g)$ for any $\widetilde{g}\in[g]$. Indeed, $||W_{g}||^2\, dv_g$ and $W_{i0j0}L^{ij}\, d{\sigma}_g$ are pointwise conformally invariant differential forms in $M^4$ and on $\Sigma^3$, respectively.
Critical points of $\mathcal{W}_{b}$ are Bach-flat metrics with vanishing $S$-tensor. In this case, we will say that the boundary is $S$-flat. The $S$-tensor is defined \cite{CG18} on the boundary $\Sigma^3$ as
\begin{align}\label{S def}
         S_{ij} := \nabla^{\alpha}W_{{\alpha}i0j} + \nabla^{\alpha}W_{{\alpha}j0i} - \nabla^0W_{0i0j} + \frac{4}{3}HW_{0i0j}.
\end{align}
The basic conformal properties of the Bach tensor and the $S$-tensor are given in the following lemma:
\begin{lemma}[\cite{CGY02}\cite{CG18}\cite{Der83}]\label{S tensor}
The Bach tensor $B_{\alpha\beta}$ and $S$-tensor $S_{ij}$ on $(M^4,\Sigma^3,g)$ satisfy the following relations:
\begin{enumerate}
    \item $B_{\alpha\beta}$ is symmetric, trace-free, divergence-free and conformally invariant in the sense that for $\widetilde{g} = e^{2w}g$,
           \[B_{\widetilde{g}} = e^{-2w}B_g.\]
    \item $S_{ij}$ is symmetric, trace-free and conformally invariant in the sense that $\widetilde{g} = e^{2w}g$,
           \[S_{\widetilde{g}} = e^{-w}S_g.\]
    \item If the boundary is totally geodesic, then 
    \[S_{ij} = \nabla^{0}P_{ij}.\]
\end{enumerate}
\end{lemma}

\subsection{Analytic properties of critical metrics}
In this subsection, we review the analytic properties of critical metrics. In particular, we discuss the smoothness of Einstein metrics and Bach-flat metrics.

An Einstein metric $g$ on a closed manifold $M^n$ satisfies an elliptic system of second order in harmonic coordinates \cite{DK81}:
\begin{align}\label{Einstein harmonic}
    \lambda{g}_{ij} = R_{ij} = -\frac{1}{2}g^{kl}\frac{\partial^2g_{ij}}{\partial{x^k}\partial{x^l}}+\cdots
\end{align}
where the dots indicate terms involving at most one derivative of the metric. From elliptic theory, it follows that the metric $g$ is analytic in harmonic coordinates and geodesic normal coordinates. Note that it is not hard to prove that a $C^{2,\alpha}$ metric $g$ satisfying Einstein condition is analytic. It is noteworthy to point out that the Einstein equation $Ric(g) = \lambda{g}$ is only {weakly} elliptic. The degeneracy is a result of the invariance of $Ric(g) = \lambda{g}$ under the action of Diff($M$), where Diff($M$) is the group of diffeomorphisms on $M$. The choice of gauge (harmonic coordinates) eliminates this degeneracy. 

\begin{remark}
The Einstein equation can be regarded as a (weakly) elliptic system of second order for the metric and the second fundamental form of the boundary can be viewed as the ``normal derivative'' of the metric on the boundary. Hence, totally geodesic boundary can be considered as vanishing Neumann boundary condition for Einstein equation. 
\end{remark}

A Bach-flat metric $g$ on a closed four-manifold $M^4$ with constant scalar curvature satisfies an elliptic system of fourth order in harmonic coordinates; see\cite{TV05a}\cite{TV05b}. Since this fact is not explicit in literature, we include calculations here. The Bach tensor can be rewritten with second Bianchi identity as \cite{Der83}
\begin{align}\label{Bach rewritten}
B_{ij} = -\frac{1}{2}\Delta{E_{ij}} + \frac{1}{6} \nabla_i\nabla_j{R} - \frac{1}{24}\Delta{R}g_{ij} - E^{kl}W_{ikjl} +E_i^kE_{jk} -\frac{1}{4}|E|^2g_{ij} + \frac{1}{6}RE_{ij}
\end{align}

If the scalar curvature is constant ($R=c$), then we have
\begin{align}
B_{ij} = -\frac{1}{2}\Delta{E_{ij}} - E^{kl}W_{ikjl} +E_i^kE_{jk} -\frac{1}{4}|E|^2g_{ij} + \frac{1}{6}cE_{ij}
\end{align}
Then Bach-flat equation in harmonic coordinates can be written as
\begin{align}
0 =  B_{ij}=  \frac{1}{4}g^{rs}g^{kl}\frac{\partial^4g_{ij}}{\partial{x^r}\partial{x^s}\partial{x^k}\partial{x^l}} + \cdots
\end{align}
where the dots indicate terms involving at most three derivatives of the metric. Hence, standard elliptic theory implies that $g$ is analytic in harmonic coordinates and geodesic normal coordinates. Note that it is not hard to prove that a $C^{4,\alpha}$ Bach-flat metric $g$ with constant scalar curvature is analytic. It is noteworthy to point out that Bach-flat equation is conformally invariant and invariant under Diff($M^4$), which leads to degeneracy of Bach-flat equations in the directions of diffeomorphisms and conformal transformations. The choices of harmonic coordinates and constant scalar curvature enable us to eliminate the degeneracy, respectively. Note that the resolution to the Yamabe problem shows that we can always find a metric of constant scalar curvature in a conformal class.

\begin{remark}
The Bach-flat equation can be regarded as a (weakly) elliptic system of fourth order for the metric and the $S$-tensor can be viewed as the ``third order normal derivative'' of the metric on the boundary, at least for the compactification of a conformally compact Einstein manifold. Hence, the vanishing $S$-tensor can be considered as a vanishing third order boundary condition for Bach-flat equation. Note that it is proper to impose two boundary conditions for an elliptic system of fourth order. The other boundary condition we impose is the umbilic condition. We point out that Bach-flatness, $S$-flatness, and umbilicity are all conformally invariant conditions.
\end{remark}

\subsection{Expansion of Riemannian metric near the boundary}
In this short subsection, we discuss the expansion of Riemannian metric near the boundary. Suppose $(M^n,\Sigma^{n-1},g)$ is a smooth manifold with boundary and $g$ is a Riemannian metric smooth up to the boundary.  Let $\{ x^i \}$ be local coordinates on $\Sigma^{n-1}$. If $r$ is the distance function to $\Sigma^{n-1}$, then we can identify a collar neighborhood of the boundary with $\Sigma^{n-1}\times{[0,\epsilon)}$, with coordinates given by $(x_i,r)$.  We want to compute the expansion of $g$ in $\Sigma^{n-1}\times{[0,\epsilon)}$.  In $\Sigma^{n-1}\times[0,\epsilon)$, write the metric $g$ as
\begin{align}
    g = dr^2 + h_{ij}(x,r)dx^idx^j,
\end{align}
where
\begin{align}
    h_{ij} = \left\langle{\partial_i,\partial_j}\right\rangle.
\end{align}

The main formulas are listed in \cite{GurGra19} and detailed calculations are given in \cite{GZ20}. We summarize these useful formulas in the following lemma.

\begin{lemma}\label{boundary expansion}
Suppose $(M^{n}, \Sigma^{n-1}, g)$ is a Riemannian manifold with boundary. Then we have the expansion for metric $g$ in $\Sigma\times[0,\epsilon)$
\begin{align}
    g = dr^2 + h_{ij}(x,r)dx^idx^j
\end{align}
where
\begin{align}
    h_{ij}(x,r) = h_{ij}^{(0)} + rh_{ij}^{(1)} + \frac{r^2}{2!}h_{ij}^{(2)} + \frac{r^3}{3!}h_{ij}^{(3)} + \frac{r^4}{4!}h_{ij}^{(4)} + O(r^5)
\end{align}
where $h_{ij}^{(k)}$ are symmetric $2$-tensors defined on $\Sigma^{n-1}$
\begin{align} \label{expansion near the boundary}
    \begin{split}
        h_{ij}^{(0)} = & \,\,\, g_{ij} \\
        h_{ij}^{(1)} = & -2L_{ij} \\
        h_{ij}^{(2)} = & -2R_{0i0j} + 2L_{ik}L_{j}^k \\
        h_{ij}^{(3)} = & -2\nabla_0{R}_{0i0j}+4L^k_{i}{R}_{j0k0} + 4L^k_{j}{R}_{i0k0}, \\
        h_{ij}^{(4)} = & -2\nabla_0\nabla_0R_{0i0j} + 6\nabla_0{R_{0i0k}}L_{j}^{k} + 6\nabla_0{R_{0j0k}}L_{i}^{k} \\
    & - 4R_{0i0k}L_{l}^{k}L_{j}^l - 4R_{0j0k}L_{l}^{k}L_{i}^l + 8R_{0i0}^k{R}_{0j0k}
    \end{split}
\end{align}
\end{lemma}

\subsection{The double of a Riemannian manifold with boundary}
We first recall the definition of the double of a smooth manifold with boundary. If $(M^n,\Sigma^{n-1})$ is a smooth manifold with boundary $\Sigma$, its double is obtained by gluing two copies of $(M^n,\Sigma^{n-1})$ together along their common boundary. Precisely speaking, the double is defined as $N := M\times \{0,1\} / \sim $ where $(x,0)\sim (x,1)$ for all $x\in\Sigma$. In this note, we shall also use the notation $N = M \bigcup_{\Sigma} M'$, where $M'$ is another copy of $M$. The double is a \emph{closed} smooth manifold. Now consider a Riemannian metric $g$ on $(M^n,\Sigma^{n-1})$. It follows that $g$ extends naturally to a metric $g_d$ on the double $N$. However, $g_d$ might be only \emph{continuous} across $\Sigma^{n-1}$. Note that regularity could only fail in the normal direction.

We have the following simple but useful lemma.
\begin{lemma}\label{C2alpha}
Suppose $(M^n,\Sigma^{n-1},g)$ is smooth up to the boundary and has totally geodesic boundary. Then the double manifold $(N,g_d)$ has $C^{2,\alpha}$ metric for any $0<\alpha<1$.
\end{lemma}
\begin{proof}
It is obvious that the regularity of $g_d$ can only fail across the boundary $\Sigma^{n-1}$. If the boundary is totally geodesic, then we have the following expansion for $g_d$ on $\Sigma\times(-r,r)$ from (\ref{expansion near the boundary}):
\begin{align}
    g_d = dr^2 + h_0 + \frac{h''}{2!}r^2 + O(r^3)
\end{align}
where $h_0 = g_d|_{\Sigma}$. It is then clear that $g_d$ is $C^{2,\alpha}$ on $N$.
\end{proof}
From Lemma \ref{C2alpha}, it is clear that the Riemannian curvature tensor is well-defined on $(N,g_d)$ with $C^\alpha$ regularity.

\section{Proofs of Theorems \ref{Connect}, \ref{hom}, and \ref{dif}}\label{1.1,1.2,1.3}
\subsection{Proof of Theorem \ref{Connect}}
In this subsection, we prove Theorem \ref{Connect}. We apply conformal deformation to obtain a metric $g_w = e^{2w}g\in[g]$ such that $(M^4,\Sigma^3,g_w)$ satisfies $\sigma_2(P_{g_w})>0$ and the boundary is totally geodesic. From there, a vanishing argument for compact Riemannian manifold with boundary will imply the result. 

\begin{proof}[Proof of Theorem \ref{Connect}]
From the proof of Theorem 1 in \cite{Chen09}, there is a metric $\hat{g} \in [g]$ such that $R_{\hat{g}} > 0$, $\sigma_2(P_{\hat{g}}) > 0$ and the boundary is minimal. Recall that we assume the boundary is umbilic and umbilicity is a conformally invariant condition. Hence, the boundary is totally geodesic. It follows from Lemma 1.2 of \cite{CGY02} in four dimensions $R_{\hat{g}} > 0$ and $\sigma_2(P_{\hat{g}}) > 0$ imply that $Ric_{\hat{g}} > 0$. Now we need the following lemma which is an analogue of Bochner's vanishing theorem for compact Riemannian manifolds with boundary.

\begin{lemma}[Proposition 3 in \cite{Wang19}]\label{Bochner}
Let $(M^n,\Sigma^{n-1},g)$ be a compact Riemannian manifold with positive Ricci
curvature and convex boundary in the sense that the second fundamental form is positive semi-definite. Then both $H^1(M,\Sigma,\mathbb{R})$ and $H^1(M,\mathbb{R})$ vanish.
\end{lemma}

For the sake of completeness, we sketch the proof of Lemma \ref{Bochner} here. 
\begin{proof}[Proof of Lemma \ref{Bochner}]
The lemma will be established by combining Hodge theory on compact manifolds with boundary and a vanishing argument. From Hodge theory:
\begin{align}
    H^1(M,\Sigma,\mathbb{R}) = \mathcal{H}^1_{R}(M),\,\,\,\, H^1(M,\mathbb{R}) = \mathcal{H}^1_{A}(M),
\end{align}
where 
\begin{align}
    \mathcal{H}^1_{R}(M) = \{\,\alpha\in\Lambda^1(M)\, : \,  d\alpha = d^*\alpha = 0,\,\,\, \alpha\wedge{n^*} = 0 \,\,\, on \,\,\, \Sigma\,  \}
\end{align}
\begin{align}
    \mathcal{H}^1_{A}(M) = \{\,\beta\in\Lambda^1(M)\, : \,  d\beta = d^*\beta = 0,\,\,\, \beta(n)= 0 \,\,\, on \,\,\, \Sigma\,  \}
\end{align}
and $n$ is the (outward) unit normal vector field on $\Sigma$ and $n^*$ is its dual $1$-form. 

Consider $\alpha\in\mathcal{H}^1_{R}(M)$. Bochner formula and some direct calculations on $\Sigma$ imply
\begin{align}
    \int_M|\nabla\alpha|^2 + Ric(\alpha,\alpha) \, dv = -\int_{\Sigma} H[\alpha(n)]^2 \, d\sigma
\end{align}
It is clear that $\alpha = 0$ if $Ric > 0$ and $H \geq 0$. Hence, ${H}^1(M,\Sigma,\mathbb{R}) = 0$.

Consider $\beta\in\mathcal{H}^1_{A}(M)$. Bochner formula and some direct calculations on $\Sigma$ imply 
\begin{align}
    \int_M|\nabla\beta|^2 + Ric(\beta,\beta) \, dv = -\int_{\Sigma} \sum_{i,j = 1}^{n-1} L_{ij}\beta(e_i)\beta(e_j) \, d\sigma
\end{align}
where $L$ is the second fundamental form of $\Sigma$ and $\{e_i\}_{i=1}^{n-1}$ is an orthonormal frame of $T\Sigma$. It is clear that $\beta = 0$ if $Ric > 0$ and $L \geq 0$. Hence, ${H}^1(M,\mathbb{R}) = 0$.
\end{proof}

It follows from Lemma \ref{Bochner} that $H^1(M^4,\Sigma^3,\mathbb{R}) = H^1(M^4, \mathbb{R}) = 0$. From the long exact sequence of the pair $(M,\Sigma)$, we have
\begin{align}
    \cdots \rightarrow H^0(M) \rightarrow H^0 (\Sigma) \rightarrow H^1(M,\Sigma) \rightarrow H^1(M) \rightarrow \cdots
\end{align}
Note that $M^4$ is assumed to be connected and thereby $H^0(M) = \mathbb{R}$. It follows that $H^0(\Sigma) = \mathbb{R}$. Therefore, the boundary $\Sigma^3$ is connected.

\end{proof}
 
 Clearly, from the proof of Lemma \ref{Bochner}, $H^1(M,\Sigma,\mathbb{R}) = 0$ can be proved under the conditions of positive Ricci curvature and nonnegative mean curvature. It is an interesting problem to ask if the conformal deformation established in \cite{Chen09} can be generalized to the case without umbilic condition on the boundary.

\subsection{Proof of Theorem \ref{hom}}
In this subsection, we prove Theorem \ref{hom}. We first apply conformal deformation to obtain a metric $g_w = e^{2w}g\in[g]$ such that $(M^4,\Sigma^3,g_w)$ satisfies $R_{g_w}>0$ and the boundary $(\Sigma^3,h_w)$ is totally geodesic. Then we consider the double manifold $N = M \bigcup_{\Sigma} {M'}$ with the metric $g_w$ which is $C^{2,\alpha}$ for any $0<\alpha<1$. We then perturb $g_w$ to construct a smooth metric $\widetilde{g}$ on $N$ such that $\widetilde{g}\in\mathcal{Y}_2^+(N)$ with $\beta(N^4,[\widetilde{g}])<8$. Theorem \ref{hom} then follows from Lemma 2.5 in \cite{CGZ} and a calculation with long exact sequences.

\begin{proof}[Proof of Theorem \ref{hom}]
From Lemma 1.1 in \cite{Esc92}, there is a metric $\hat{g} \in [g]$ such that $R_{\hat{g}} > 0$ and the boundary is minimal. Recall that we assume the boundary is umbilic and umbilicity is a conformally invariant condition. Hence, the boundary is totally geodesic. It follows from Lemma \ref{C2alpha} that the double manifold $(N,g_d)$ is a closed Riemannian manifold with $C^{2,\alpha}$ metric. Hence, the Riemannian curvature tensor is well-defined for $g_d$ and $(N^4,g_d)$ satisfies $0\leq\beta(N^4,[g_d]) <8$. With same argument in \cite{CQY04}, we may smooth the metric $g_d$ if necessary such that there is a smooth metric $\widetilde{g}$ satisfying $0\leq\beta(N^4,[\widetilde{g}]) <8$. From Lemma 2.5 in \cite{CGZ}, we have $b_2(N) = 0$ or $b_2(N) = 1$. We now rule out the possibility of $b_2(N) = 1$ by establishing the following lemma.
\begin{lemma}\label{topo}
Suppose $(N^4,\widetilde{g})$ is constructed as above. Then $b_2(N) = 0$.
\end{lemma}
\begin{proof}[Proof of Lemma \ref{topo}]
From Theorem \ref{Connect}, we have $H^1(M) = H^1(M,\Sigma) = 0$. From Corollary F in \cite{Gur98}, it follows that $H^1(N) = 0$. Note that Poincar\`e duality implies that $H^3(N) = 0$. Consider the long exact sequence for $N = M \bigcup_{\Sigma} M'$ and note that $M'$ is another copy of $M$:

\begin{align}\label{exact}
     H^1(M)\,\oplus\,H^1(M) \rightarrow H^1(\Sigma) \xrightarrow{i} H^2(N) \xrightarrow{j} H^2(M)\,\oplus\,H^2(M) \xrightarrow{k} H^2(\Sigma) \rightarrow H^3(N)
\end{align}
Note that $H^1(M)\,\oplus\,H^1(M) = H^3(N) =0$. From basic linear algebra, we have 
\begin{align}\label{Linear}
    H^2(N) = \ker{j}\oplus\Ima{j},\,\,\, H^2(M)\,\oplus\,H^2(M) = \ker{k}\oplus\Ima{k}.
\end{align}
From exactness, we have $i$ is injective and $k$ is surjective. In addition, we have
\begin{align}\label{Exact}
    \ker{j} = \Ima{i} = H^1(\Sigma),\,\, \ker{k} = \Ima{j},\,\, \Ima{k} = H^2(\Sigma).
\end{align}
Note that $\Sigma$ is a closed $3$-manifold and Poincar\'e duality thereby implies that $H^1(\Sigma) = H^2(\Sigma)$. Combining (\ref{Linear})(\ref{Exact}) and $H^1(\Sigma) = H^2(\Sigma)$, we obtain
\begin{align}
    H^2(N) = \ker{j}\oplus\Ima{j} = H^1(\Sigma)\oplus\Ima{j} = H^2(\Sigma)\oplus\ker{k} = H^2(M)\,\oplus\,H^2(M).
\end{align}
It follows that $b_2(N)$ is even and thereby $b_2(N) = 0$.
\end{proof}
Now we continue to prove Theorem \ref{hom}. 
Plug $H^2(N) = H^2(M)\,\oplus\,H^2(M) = 0$ into (\ref{exact}). 
We have $H^1(\Sigma) = H^2(\Sigma) =0$. From Theorem \ref{Connect}, $H^0(\Sigma) = H^3(\Sigma) = \mathbb{R}$. 
Hence, $\Sigma$ is a homology $3$-sphere. Note that the double manifold $N = M \bigcup_{\Sigma} M'$ admits a metric $\widetilde{g} \in \mathcal{Y}_2^+(N)$. Hence, Corollary B in \cite{CGY02} implies that $N$ admits a metric with positive Ricci curvature and thereby the universal cover $\widetilde{N}$ of $N$ is compact by Bonnet-Myers theorem. It is now clear that $b_2(\widetilde{N}) = 0$ and $\widetilde{N}$ is simply-connected. The celebrated work of Freedman \cite{Fre82} implies that $\widetilde{N}$ is homeomorphic to $S^4$.  Recall $(M^4,\Sigma^3)$ is assumed to be oriented. Hence, $N$ itself is homeomorphic to $S^4$. The homology of $M^4$ can now be calculated from the long exact sequence for $N = M \bigcup_{\Sigma} M'$ and the conclusion follows easily.
\end{proof}
\begin{remark}
Note that $0\leq\beta(N^4,[g])<8$ for a \emph{closed} Riemannian manifold $(N^4,g)$ implies that $b_2(N^4) \leq 1$ from Lemma 2.5 of \cite{CGZ}. In addition, $b_2(N^4) = 0$ and $b_2(N^4) = 1$ are realized by $(S^4,[g_{S^4}])$ and $(\mathbb{CP}^2,[g_{FS}])$, respectively. However, in the case of manifolds with boundary, if the boundary is assumed to be umbilic, the range $0\leq\beta_b(M^4,\Sigma^3,[g])<8$ implies that the double manifold must be $S^4$. Hence, $\mathbb{CP}^2$ cannot be realized as the double of any manifold with umbilic boundary under the condition $0\leq\beta_b(M^4,\Sigma^3,[g])<8$. This interesting phenomenon shows that the umbilic condition imposes additional symmetry on the double manifold.
\end{remark}

\subsection{Proof of Theorem \ref{dif}}
In this subsection, we prove Theorem \ref{dif}. We first apply conformal deformation to obtain a metric $g_w = e^{2w}g\in[g]$ such that $(M^4,\Sigma^3,g_w)$ satisfies $R_{g_w}>0$ and the boundary $(\Sigma^3,h_w)$ is totally geodesic. Then we consider the double manifold $N = M \bigcup_{\Sigma} {M'}$ with the metric $g_w$ which is $C^{2,\alpha}$ for any $0 < \alpha < 1$. We then perturb $g_w$ to construct a smooth metric $\widetilde{g}$ on $N$ such that $\widetilde{g}\in\mathcal{Y}_2^+(N)$ with $\beta(N^4,[\widetilde{g}])<4$. Theorem \ref{dif} then follows from the proof of  Theorem A in \cite{CGY03}.
\begin{proof}[Proof of Theorem \ref{dif}]
From Lemma 1.1 of \cite{Esc92}, there is a metric $\hat{g} \in [g]$ such that $R_{\hat{g}} > 0$ and the boundary is minimal. Recall that we assume the boundary is umbilic and umbilicity is a conformally invariant condition. Hence, the boundary is totally geodesic. From Lemma \ref{C2alpha}, the double $(N,g_d)$ is a closed Riemannian manifold with $C^{2,\alpha}$ metric. Hence, the Riemannian curvature tensor is well-defined for $g_d$ and $(N^4,g_d)$ satisfies $0\leq\beta(N^4,[g_d])<4$. With same argument in \cite{CQY04}, we may smooth the metric $g_d$ if necessary such that there is a smooth metric $\widetilde{g}$ satisfying $0\leq\beta(N^4,[\widetilde{g}]) <4$. Then it is easy to apply the arguments in \cite{CGY02}\cite{CGY03} by using Ricci flow to deform $(N^4,\widetilde{g})$ to the round $(S^4,g_{S^4})$ with the doubling property preserved all the way. For a detailed description of this process, see the last section of \cite{CQY04}. It easily follows that $M^4$ is diffeomorphic to $B^4$ and $\Sigma^3$ is diffeomorphic to $S^3$.
\end{proof}

\section{Proofs of Theorems \ref{beta b sphere}, \ref{conformal E sphere}, \ref{beta b and Weyl sphere}, and \ref{beta b 8}}\label{1.4,1.5,1.6,1.7}
\subsection{Proof of Theorem \ref{beta b sphere}}
In this subsection, we prove Theorem \ref{beta b sphere}. We shall establish the regularity of the metric on a double manifold arising from a Bach-flat four-manifold with boundary such that the boundary is umbilic and $S$-flat. Theorem \ref{beta b sphere} then follows from Theorem A.

Note that the second fundamental form appears in the expansion of metric (\ref{expansion near the boundary}) as the first order normal derivative. It turns out that the $S$-tensor is (up to a constant multiple) the third order normal derivative for a metric with constant scalar curvature and totally geodesic boundary. We start by the following lemma for manifolds with totally geodesic boundary.

\begin{lemma}\label{curvature tot geo}
Suppose $(M^4,\Sigma^3,g)$ is a smooth Riemannian manifold with totally geodesic boundary. Then we have on $\Sigma^3$
\begin{align}
    R_{j0} = 0,\,\,\,\,\,\, P_{j0} = 0, \,\,\,\,\,\,W_{ki0j} = 0,
\end{align}
and 
\begin{align}
    S_{ij} = \nabla^0P_{ij} = \nabla^0W_{0i0j}.
\end{align}
\end{lemma}
\begin{proof}
Recall the Gauss equations and Codazzi equations on $\Sigma^3$:
\begin{align}\label{Gauss}
    R_{ikjl} = R^{\Sigma}_{ikjl} - L_{ij}L_{kl} + L_{il}L_{jk},
\end{align}
\begin{align}\label{Codazzi}
    R_{ijk0} = -\nabla_{j}^{\Sigma}L_{ik} + \nabla_{i}^{\Sigma}L_{jk}.
\end{align}
For totally geodesic boundary, we have $L_{ij} \equiv 0$ and thereby on $\Sigma^3$:
\begin{align}\label{Gauss tot geo}
    R_{ikjl} = R^{\Sigma}_{ikjl}
\end{align}
\begin{align}\label{Codazzi tot geo}
    R_{ijk0} = 0
\end{align}
Taking the trace of (\ref{Codazzi tot geo}), we obtain on $\Sigma^3$
\begin{align}\label{R_0j tot geo}
    R_{0j} = 0,
\end{align}
which implies that on $\Sigma^3$ 
\begin{align}\label{P_0j tot geo}
    P_{0j} = 0
\end{align} 
since $g_{0j} = 0$ on $\Sigma^3$. Hence, we have the following decomposition of $R_{ki0j}$ on $\Sigma^3$:
\begin{align}
    {R}_{ki0j} = {W}_{ki0j} + g_{k0}P_{ij} + g_{ij}P_{k0} - g_{kj}P_{i0} - g_{i0}P_{kj} = W_{ki0j}.
\end{align}
Combining this equation with (\ref{Codazzi tot geo}), we have on $\Sigma^3$
\begin{align}\label{W_ki0j on boundary}
    {W}_{ki0j} = 0.
\end{align}
Note that (\ref{P_0j tot geo}) and (\ref{W_ki0j on boundary}) imply that for totally geodesic boundary
\begin{align}
    \nabla^iP_{0j} = \nabla^k{{W}_{k{i0j}}} = 0.
\end{align}
In fact, we may calculate at $p\in\Sigma^3$ in Fermi coordinates. We have $g_{\alpha\beta} = \delta_{\alpha\beta}$ and $\Gamma_{\alpha\beta}^{\gamma} = 0$ at $p\in\Sigma^{3}$ since the boundary is totally geodesic. Thus, we have at $p\in\Sigma$
\begin{align}
    \nabla^iP_{0j} = \partial_{i}P_{0j} - \Gamma_{i0}^{\alpha}P_{{\alpha}j} - \Gamma_{ij}^{\alpha}P_{0{\alpha}} = 0.
\end{align}
$\nabla^k{{W}_{k{i0j}}} = 0$ on $\Sigma^3$ follows similarly.

Hence, we have by the Bianchi identities on $\Sigma^3$
\begin{align}
    \nabla^0P_{ij} = \nabla^0P_{ij} - \nabla^iP_{0j} = \nabla^{\alpha}W_{\alpha{i0j}} = \nabla^{0}W_{0{i0j}}.
\end{align}
By (3) of Lemma {\ref{S tensor}}, we have on $\Sigma^3$
\begin{align}
    S_{ij} = \nabla^0W_{0i0j}.
\end{align}
\end{proof}

We establish the following lemma which is of independent interest. This lemma reveals the relationship between the $S$-tensor and the third order normal derivative of a metric with constant scalar curvature and totally geodesic boundary.

\begin{lemma}\label{S equal to h3}
Suppose $(M^4,\Sigma^3,g)$ is a smooth Riemannian manifold with constant scalar curvature and totally geodesic boundary. Then we have on $\Sigma^3$
\begin{align}
    h'''_{ij} = -4S_{ij}.
\end{align}
\end{lemma}

\begin{proof}
We calculate $h'''_{ij}$ on $\Sigma^3$ by (\ref{expansion near the boundary}):
\begin{align}\label{h'''}
    h'''_{ij} = -2\nabla^0{R}_{0i0j}+8L^k_{(i}{R}_{j)0k0}.  
\end{align}
For totally geodesic boundary, we have $L_{ij} = 0$ and thereby 
\begin{align}\label{h''' tot geo}
    h'''_{ij} = -2\nabla^0{R}_{0i0j}.
\end{align}
Note that $g_{0j} = 0$ and $g_{00} = 1$. We have the decomposition of curvature tensor:
\begin{align}
    {R}_{0i0j} = {g}_{00}{P}_{ij} + {g}_{ij}{P}_{00} - {g}_{0j}{P}_{i0} - {g}_{i0}{P}_{0j} + {W}_{0i0j} = {P}_{ij} + {g}_{ij}{P}_{00} + {W}_{0i0j}.
\end{align}
Taking covariant derivative in normal direction, we obtain on $\Sigma^3$
\begin{align}\label{normal Riem}
    \nabla^0{R}_{0i0j} = \nabla^0{W}_{0i0j} + \nabla^0P_{ij} + {g}_{ij}\nabla^0P_{00}.
\end{align}
By Lemma \ref{curvature tot geo}, we have on $\Sigma^3$
\begin{align}\label{normal R_0i0j on boundary}
    \nabla^0{R}_{0i0j} = 2S_{ij} + {g}_{ij}\nabla^0P_{00}.
\end{align}
We now calculate $\nabla^0P_{00}$ on $\Sigma^3$. Note that by definition
\begin{align}
    g^{\alpha\beta}P_{\alpha\beta} = \frac{1}{6}R.
\end{align}
Since $R$ is a constant, $g_{00} = 1$ and $g_{0j} = 0$, we take covariant derivative in normal direction to obtain on $\Sigma^3$
\begin{align}\label{normal R on boundary}
    \nabla^0P_{00} + g^{ij}\nabla^0P_{ij} = 0. 
\end{align}
By Lemma \ref{S tensor}, we have for totally geodesic boundary $S_{ij} = \nabla^0P_{ij}$ and $S$-tensor is trace-free. Therefore, we have on $\Sigma^3$
\begin{align}
    g^{ij}\nabla^0P_{ij} = 0.
\end{align}
Combining this equation with (\ref{normal R on boundary}), we obtain on $\Sigma^3$
\begin{align}\label{normal P_00 on boundary}
    \nabla^0P_{00} = 0. 
\end{align}
Putting (\ref{h''' tot geo})(\ref{normal R_0i0j on boundary}) and (\ref{normal P_00 on boundary}) together, we derive $h'''_{ij} =  -4 S_{ij}$. 
\end{proof}
With Lemma \ref{S equal to h3}, we are ready to prove Theorem \ref{beta b sphere}.
\begin{proof}[Proof of Theorem \ref{beta b sphere}]
Consider the Yamabe metric $g_Y\in[g]$ such that  $(M^4,\Sigma^3, g_Y)$ has constant scalar curvature and minimal boundary. From the assumption that the boundary is umbilic, minimality implies that the boundary is totally geodesic, i.e. $L_{ij} = 0$. In addition, by (\ref{expansion near the boundary}) and Lemma \ref{S equal to h3}, $S_{ij}=0$ implies that the expansion of $g_Y$ near the boundary $\Sigma^3$ has the form:
\begin{align}\label{Expansion beta}
g_Y = dr^2 + h + \frac{h^{(2)}}{2!}r^2 + \frac{h^{(4)}}{4!}r^4 + O(r^5)
\end{align}

Now consider the double manifold $N = M\bigcup_{\Sigma} M'$, where $M'$ is another copy of $M$.  The metric $g_Y$ naturally extends to a metric $g_d$ on $N$. From (\ref{Expansion beta}), it follows that $g_d$ is of class $C^{4,\alpha}$ for any $0<\alpha<1$. Note that $g_d$ is Bach-flat since $g_Y$ is Bach flat and Bach tensor is well-defined for $C^{4,\alpha}$ metric. It follows that $g_d$ is smooth from the discussion in Section 2.3. The condition $0 \leq \beta(M^4,\Sigma^3,[g]) < 4$ implies $\beta(N,[g_d]) <4$. It follows from Theorem A that $(M^4,g_{d})$ is conformally equivalent to $(S^4,g_{S^4})$. Now it is clear that $(M^4, \Sigma^3, g)$ is conformally equivalent to $(S^4_+,S^3,g_{S^4_+})$.
\end{proof}

\subsection{Proof of Theorem \ref{conformal E sphere}}
In this subsection, we prove Theorem \ref{conformal E sphere}. The proof relies on the expansion of Riemannian metric near the boundary and an application of Theorem C to the double manifold.
\begin{proof}[Proof of Theorem C]
Similar to the argument in the proof of Theorem \ref{beta b sphere}, we consider the double manifold $(N,g_d)$ arising from $(M^4,\Sigma^3,g_Y)$. Note that 
\begin{align}
\mathcal{E}([g]) = \int_M\sigma_2(P_{g_Y}) \, \, dv_{g_Y}
\end{align}
since $(M^4,\Sigma^3,g_Y)$ has totally geodesic boundary. Hence, $(N,g_d)$ is a closed Bach-flat four-manifold with 
\begin{align}
\int_N\sigma_2(P_{g_d}) \, dv_{g_d} = 2\mathcal{E}([g]) 
\end{align}
Taking $\epsilon_1 = \epsilon$ where $\epsilon$ is the constant in Theorem C, we have 
\begin{align}
    \int_N\sigma_2(P_{g_d}) \, dv_{g_d} = 2\mathcal{E}([g]) \geq 4(1-\epsilon)\pi^2.
\end{align}
Theorem \ref{conformal E sphere} now follows from Theorem C.
\end{proof}

\subsection{Proof of Theorem \ref{beta b and Weyl sphere}}
In this subsection, we prove Theorem \ref{beta b and Weyl sphere}. The proof relies on an upper bound for $\int_{N}\sigma_2(P_g)\,dv_{g}$ on a closed four-manifold $(N^4,g)$ proved by M. Gursky in \cite{Gur97}:
\begin{lemma}[\cite{Gur97}]\label{Gursky 97}
Suppose $(N^4,g)$ is a closed Riemannian four-manifold with $Y(N^4, [g]) \geq 0$. Then 
\begin{align}
\int_{N}\sigma_2(P_g)\, dv_g \leq 4\pi^2,
\end{align}
where equality holds if and only if $(N^4,g)$ is conformally equivalent to $(S^4,g_{S^4})$.
\end{lemma}
With this lemma at hand, we now prove Theorem \ref{beta b and Weyl sphere}.
\begin{proof}[Proof of Theorem \ref{beta b and Weyl sphere}]
Similar to the argument in the proof of Theorem \ref{beta b sphere}, we consider the double manifold $(N,g_d)$ arising from $(M^4,\Sigma^3,g_Y)$. From Corollary F of \cite{Gur98}, we have $b_1(N) = 0$ and Poincar\'e duality easily implies $b_3(N) = 0$ and thereby $\chi(N) = 2 + b_2(N)$. Chern-Gauss-Bonnet formula reads
\begin{align}\label{GBC}
8\pi^2\chi(N) = \int_N||W_{g_d}||^2\, dv_{g_d} + 4\int_N\sigma_2(P_{g_d})\, dv_{g_d}.
\end{align}
Note that 
\begin{align}\label{Weyl g_d}
\int_N||W_{g_d}||^2\, dv_{g_d} = 2\int_M||W_{g_{Y}}||^2\, dv_{g_{Y}} <8\pi^2.
\end{align}
Combining (\ref{GBC})(\ref{Weyl g_d}) and Lemma \ref{Gursky 97}, we obtain
\begin{align}
8\pi^2\chi(N) < 16\pi^2 + 8\pi^2 = 24\pi^2
\end{align}
Therefore, $\chi(N) = 2 + b_2(N) < 3$ and thereby $b_2(N) = 0$ and $\chi(N) = 2$. Returning to (\ref{GBC}), we have
\begin{align}
4\int_N\sigma_2(P_{g_d})\, dv_{g_d}  = 16\pi^2 - \int_N||W_{g_d}||^2\, dv_{g_d} > 8\pi^2
\end{align}
Note that the boundary of $(M^4,\Sigma^3,g_{Y})$ is totally geodesic and thereby
\begin{align}
\mathcal{E}([g]) = \int_M\sigma_2(P_{g_Y}) \, dv_{g_Y}  = \frac{1}{2}\int_N\sigma_2(P_{g_d})\, dv_{g_d} > \pi^2
\end{align}
Therefore, we derive
\begin{align}
\beta_b(M^4,\Sigma^3,[g]) = \frac{\int_M||W_g||^2\, dv_g}{\mathcal{E}([g])} <\frac{4\pi^2}{\pi^2} = 4.
\end{align}
Theorem \ref{beta b and Weyl sphere} now follows from Theorem \ref{beta b sphere}.
\end{proof}
\subsection{Proof of Theorem \ref{beta b 8}}
In this subsection, we prove Theorem \ref{beta b 8}. There are two important ingredients in the proof. The first ingredient is a characterization of $(S^2 \times S^2,g_{{S^2}\times{S^2}})$ established in \cite{Z18}.

\begin{lemma}\label{Z}
Let $(N^4,g)$ be a closed Bach-flat manifold with $b_2^+(N^4)=b_2^-(N^4)>0$. There is an $\epsilon > 0$ such that if $g\in{\mathcal{Y}_2^+(N^4)}$ with
\begin{equation}\label{pinching S2S2}
8\leq\beta(N^4,[g])<8(1+\epsilon),
\end{equation}
then $(N^4,g)$ is conformally equivalent to $(S^2 \times S^2,g_{{S^2}\times{S^2}})$, where $g_{{S^2}\times{S^2}}$ is the standard product metric. In fact, $(N,g_{Y})$ is isometric to $(S^2 \times S^2,g_{{S^2}\times{S^2}})$.
\end{lemma}

The second ingredient is a topological lemma which is a generalization of Lemma \ref{topo}.

\begin{lemma}\label{TOP}
Suppose $(M^4, \Sigma^3)$ is a manifold with boundary and $N = M\bigcup_{\Sigma} M'$ is the double of $(M^4, \Sigma^3)$. In addition, we assume $H^1(M) = H^1(M,\Sigma) = H^1(N) = 0$. Then $b_2(N)$ is even.
\end{lemma}

\begin{proof}[Proof of Lemma \ref{TOP}]
Note that $H^3(N) = 0$ by Poincar\'e duality. Consider the long exact sequence for $N = M \bigcup_{\Sigma} M'$ and note that $M'$ is another copy of $M$:
\begin{align}
     H^1(M)\,\oplus\,H^1(M) \rightarrow H^1(\Sigma) \xrightarrow{i} H^2(N) \xrightarrow{j} H^2(M)\,\oplus\,H^2(M) \xrightarrow{k} H^2(\Sigma) \rightarrow H^3(N)
\end{align}
Note that $H^1(M)\,\oplus\,H^1(M) = H^3(N) =0$. From basic linear algebra, we have 
\begin{align}
    H^2(N) = \ker{j}\oplus\Ima{j},\,\,\, H^2(M)\,\oplus\,H^2(M) = \ker{k}\oplus\Ima{k}.
\end{align}
From exactness, we have $i$ is injective and $k$ is surjective. In addition, we have
\begin{align}
    \ker{j} = \Ima{i} = H^1(\Sigma),\,\, \ker{k} = \Ima{j},\,\, \Ima{k} = H^2(\Sigma).
\end{align}
Note that $\Sigma$ is a closed $3$-manifold and Poincar\'e duality thereby implies that $H^1(\Sigma) = H^2(\Sigma)$. Combining (\ref{Linear})(\ref{Exact}) and $H^1(\Sigma) = H^2(\Sigma)$, we derive
\begin{align}
    H^2(N) = \ker{j}\oplus\Ima{j} = H^1(\Sigma)\oplus\Ima{j} = H^2(\Sigma)\oplus\ker{k} = H^2(M)\,\oplus\,H^2(M).
\end{align}
It then follows that $b_2(N)$ is even.
\end{proof}

We are now at the position to prove Theorem \ref{beta b 8}.

\begin{proof}[Proof of Theorem \ref{beta b 8}]
By Theorem \ref{hom} we assume without loss of generality that
\begin{align}
    8 \leq \beta_b(M^4,\Sigma^3,[g]) < 8(1+\epsilon_2).
\end{align}
Similar to the argument in the proof of Theorem \ref{beta b sphere}, we consider the double manifold $(N,g_d)$ arising from $(M^4,\Sigma^3,g_Y)$. 
Similar to the proof of Theorem 1.2, we have $b_1(N) = b_3(N) = 0$ and $\chi(N) = 2 + b_2(N)$. By Theorem \ref{beta b sphere}, we have $H^1(M) = H^1(M,\Sigma) = 0$. It then follows from Lemma \ref{TOP} that $b_2(N)$ is even. Chern-Gauss-Bonnet formula reads
\begin{align}
8\pi^2\chi(N) = \int_N \, ||W_{g_d}||^2 \, \, dv_{g_d} + 4\int_N \, \sigma_2(g_d) \, \, dv_{g_d}.
\end{align}
By Lemma \ref{Gursky 97} we have
\begin{align}
    4\int_N \, \sigma_2(g_d) \, \, dv_{g_d} \leq 16\pi^2 
\end{align} 
Set $\beta_b(M^4,\Sigma^3,g) = 8(1+\epsilon_2)$. Then
\begin{align}
    8\pi^2\chi(N) = (12+8\epsilon_2) \int_N \, \sigma_2(g_d) \, \, dv_{g_d} \leq (3+2\epsilon_2) 16\pi^2
\end{align}
Recall $\chi(N) = 2 + b_2(N)$. Hence, $2 + b_2(N) \leq 6 + 4\epsilon_2$. If we assume $\epsilon_2 < \frac{1}{4}$, then $b_2(N) \leq 4$ since $b_2(N)$ is an integer.
Since $b_2(N)$ is even, we only need to consider three cases $b_2(N) = 0$, $b_2(N) = 2$, and $b_2(N) = 4$. We shall show that only $b_2(N) = 0$ can happen if $\epsilon_2$ is chosen to be sufficiently small.

\textbf{Case 1: $b_2(N) = 0$.} Note that the double manifold $N = M \bigcup_{\Sigma} M'$ admits a metric $\widetilde{g} \in \mathcal{Y}_2^+(N)$. Hence, Corollary B of \cite{CGY02} implies that $N$ admits a metric with positive Ricci curvature and thereby the universal cover $\widetilde{N}$ of $N$ is compact by Bonnet-Myers theorem. It is now clear that $b^2(\widetilde{N}) = 0$ and $\widetilde{N}$ is simply-connected. The celebrated work of Freedman \cite{Fre82} implies that $\widetilde{N}$ is homeomorphic to $S^4$.  Recall $(M^4,\Sigma^3)$ is assumed to be oriented. Hence, $N$ itself is homeomorphic to $S^4$. The homology of $M^4$ can now be calculated from the long exact sequence for $N = M \bigcup_{\Sigma} M'$ and the conclusion follows easily.

\textbf{Case 2: $b_2(N) = 2$.} Then $\chi(N) = 4$. There are two sub-cases to consider.

Sub-case 1: $b_2^+(N) = 2$ and $b_2^-(N) = 0$. Hence, we have $\tau(N) = 2$. Signature formula reads
\begin{align}
    24\pi^2 = 12\pi^2\tau(N) = \int_N \, ||W^+_{g_d}||^2 \, \, dv_{g_d} - \int_N \, ||W^-_{g_d}||^2 \, \, dv_{g_d},
\end{align}
which implies
\begin{align}\label{Weyl L2}
    \int_N \, ||W_{g_d}||^2 \, \, dv_{g_d} = \int_N \, ||W^+_{g_d}||^2 \, \, dv_{g_d} + \int_N \, ||W^-_{g_d}||^2 \, \, dv_{g_d} \geq 24\pi^2.
\end{align}
Note that
\begin{align}\label{beta gd}
    \beta_b(M^4,\Sigma^3,g) = \dfrac{\int_M||W_{g_d}||^2\, dv_{g_d}}{\int_M{\sigma_2(P_{g_d})}\, dv_{g_d}} = 8(1+\epsilon_2)
\end{align}
Hence, combining (\ref{Weyl L2}) and (\ref{beta gd}) we have
\begin{align}
        \int_N \, \sigma_2(g_d) \, \, dv_{g_d} \geq \frac{3}{1+\epsilon_2}\pi^2.
\end{align}

Plugging these two inequalities back into Chern-Gauss-Bonnet formula, we derive
\begin{align}
    32\pi^2 = 8\pi^2\chi(N) = \int_N \, ||W_{g_d}||^2 \, \, dv_{g_d} + 4\int_N \, \sigma_2(g_d) \, \, dv_{g_d} \geq \left(24 + \frac{12}{1+\epsilon_2}\right)\pi^2.
\end{align}
If we assume $\epsilon_2 < \frac{1}{2}$, this inequality cannot hold. Therefore, $b_2^+(N) = 2$ and $b_2^-(N) = 0$ cannot happen.

Sub-case 2: $b_2^+(N) = 1$ and $b_2^-(N) = 1$. It follows from Lemma \ref{Z} that $(N,g_d)$ is isometric to $(S^2 \times S^2, g_{{S^2}\times{S^2}})$ provided that $\epsilon_2$ is chosen sufficiently small. In this case, $(\Sigma^3,h)$ is a totally geodesic hypersurface in $(S^2 \times S^2, g_{{S^2}\times{S^2}})$, which contradicts the classification of totally geodesic submanifolds in $(S^2 \times S^2, g_{{S^2}\times{S^2}})$. Indeed, Corollary 3.1 of \cite{ChenNagano} implies that there is no totally geodesic hypersurface in $(S^2 \times S^2, g_{{S^2}\times{S^2}})$. Therefore, $b_2^+(N) = 1$ and $b_2^-(N) = 1$ cannot happen.

\textbf{Case 3: $b_2(N) = 4$.} Then $\chi(N) = 6$ and Chern-Gauss-Bonnet formula reads
\begin{align}\label{CGB b_2 4}
    48\pi^2 = 8\pi^2\chi(N) = \int_N \, ||W_{g_d}||^2 \, \, dv_{g_d} + 4\int_N \, \sigma_2(g_d) \, \, dv_{g_d}.
\end{align}
Note that
\begin{align}\label{beta b last}
    \beta_b(M^4,\Sigma^3,g) = \dfrac{\int_M||W_{g_d}||^2\, dv_{g_d}}{\int_M{\sigma_2(P_{g_d})}\, dv_{g_d}} = 8(1+\epsilon_2).
\end{align}
Hence, we obtain by combining (\ref{CGB b_2 4}) and (\ref{beta b last})
\begin{align}
    \int_N \, \sigma_2(g_d) \, \, dv_{g_d} = \frac{24}{6 + 4\epsilon_2}\pi^2.
\end{align}
It then follows from Theorem C that $(N^4,g_d)$ is conformally equivalent to $(S^4,g_{S^4})$ if $\epsilon_2$ is chosen sufficiently small since $(N^4,g_{d})$ is Bach-flat, which  contradicts the fact $b_2(N) = 4$. Therefore, $b_2(N) = 4$ cannot happen.
\end{proof}

\section{Further remarks}\label{final}
In this section, we have a concise discussion of the moduli spaces of the critical metrics on four-manifolds with boundary. On {closed} four-manifolds, the moduli spaces of Einstein metrics and Bach-flat metrics have been studied extensively in \cite{And}\cite{BKN89}\cite{TV05a}\cite{TV05b}. In addition, diffeomorphic finiteness for Einstein $n$-manifolds with bounded $L^{n/2}$ norm of curvature has been proved by M. Anderson and J. Cheeger in \cite{AC91}. In fact, they established a more general result under the condition of bounded Ricci curvature and bounded $L^{n/2}$ norm of curvature. For Bach-flat four-manifolds of positive Yamabe type with uniform lower bound for total integral of $\sigma_2$-curvature and bounded $L^{2}$ norm of Weyl curvature, diffeomorphic finiteness has been proved by S.-Y. A. Chang, J. Qing, and P. Yang in \cite{CQY}. These two diffeomorphic finiteness theorems are both established by the construction of bubble trees based on the precise study of moduli spaces of critical metrics. 

For the moduli spaces of the critical metrics on four-manifolds with boundary considered in this note, the additional difficulty is how to analyze the (possible) degeneration behaviour of points on the boundary. However, we may consider the double manifold $(N^4,g_{d})$ which is a closed manifold with a {smooth} Riemannian metric from the discussion in Section \ref{prelim}. In this way, boundary points can be ``transformed'' as interior points and analysis on closed manifolds can be applied. Based on this observation, we similarly establish compactness and diffeomorphic finiteness theorems for critical metrics on four-manifolds with boundary. The details will appear somewhere else.

\bibliographystyle{amsplain}

\end{document}